\documentclass[a4paper, 12pt]{article}
\usepackage{amsmath}
\usepackage{amssymb,amsthm}
\usepackage[runin]{abstract}

\usepackage{enumitem}
\setenumerate[1]{itemsep=0pt,partopsep=0pt,parsep=\parskip,topsep=5pt}
\setitemize[1]{itemsep=0pt,partopsep=0pt,parsep=\parskip,topsep=5pt}
\setdescription{itemsep=0pt,partopsep=0pt,parsep=\parskip,topsep=5pt}

\abslabeldelim{\ \ }
\newtheorem{theorem}{Theorem}[section]

\newtheorem{lemma}[theorem]{Lemma}

\newtheorem{proposition}[theorem]{Proposition}

\numberwithin{equation}{section}
\title{Normalized solution to coupled nonhomogeneous nonlinear elliptic system with three wave interaction under unbounded potentials\thanks{Supported by NSFC (12171326, 12171014, 12001382) and KZ202010028048.}}

\author{Mingyang Han\thanks{Corresponding author. Email: myanghan@163.com.}\\ School of Mathematical Sciences, Capital Normal University}
\date{}
\begin{document}
	\maketitle 
	\numberwithin{equation}{section}
	\setcounter{section}{0}
	 \begin{abstract}
	    In this paper we use variational approach to find the normalized solution of the system
		\begin{equation*}
		\begin{cases} - \Delta u + \left( {{V_1}(x) + {\lambda _1}} \right)u = {\mu _1}{{\left| u \right|}^{{p} - 2}}u + \beta vw, & \text { in }\ {\mathbb{R}^3}, \\ 
		 - \Delta v + \left( {{V_2}(x) + {\lambda _2}} \right)v = {\mu _2}{{\left| v \right|}^{{p} - 2}}v + \beta uw, & \text { in }\ {\mathbb{R}^3}, \\ 
		 - \Delta w + \left( {{V_3}(x) + {\lambda _3}} \right)w = {\mu _3}{{\left| w \right|}^{{p} - 2}}w + \beta uv, & \text { in }\ {\mathbb{R}^3},
					\end{cases}
		\end{equation*}
		with the condition
		\begin{equation*}
			\int_{{\mathbb{R}^3}} {u^2 = a} ,\int_{{\mathbb{R}^3}} {v^2 = b} ,\int_{{\mathbb{R}^3}} {w^2 = c},
		\end{equation*}
		where $\mu_{1},\mu_{2},\mu_{3},a,b,c>0$, $\beta >0$, $2<p<10/3$ and $\lambda_1,\lambda_2,\lambda_3 \in \mathbb{R}$ are Lagrangian multipliers. We prove the existence of minimizer for the system with prescribed $L^2$-norm under some conditions on $V_i(x),i=1,2,3.$
 	\end{abstract}
 
	\paragraph{Keywords}  Normalized solution; three wave interaction; unbounded potentials	
	\paragraph{2000 MR Subject Classification} $\ \ $35J20; 35J50; 35J57
	\section{Introduction} 

	In this paper we seek for the normalized solution of the system
	\begin{equation}\label{key1}
		\begin{cases} - \Delta u + \left( {{V_1}(x) + {\lambda _1}} \right)u = {\mu _1}{{\left| u \right|}^{{p} - 2}}u + \beta vw, & \text { in }\ {\mathbb{R}^3}, \\ 
		 - \Delta v + \left( {{V_2}(x) + {\lambda _2}} \right)v = {\mu _2}{{\left| v \right|}^{{p} - 2}}v + \beta uw,& \text { in }\ {\mathbb{R}^3}, \\ 
		 - \Delta w + \left( {{V_3}(x) + {\lambda _3}} \right)w = {\mu _3}{{\left| w \right|}^{{p} - 2}}w + \beta uv,& \text { in }\ {\mathbb{R}^3}, 
			\end{cases}
	\end{equation}
	with the condition
	\begin{equation}\label{key2}
		\int_{{\mathbb{R}^3}} {u^2 = a} ,\int_{{\mathbb{R}^3}} {v^2 = b} ,\int_{{\mathbb{R}^3}} {w^2 = c},
	\end{equation}
	where $\mu_{1},\mu_{2},\mu_{3},a,b,c>0$, $\beta >0$, $2<p<10/3$ and $\lambda_1,\lambda_2,\lambda_{3}$ are Lagrangian multipliers will be determined.
	We prove the system (\ref{key1})-(\ref{key2}) has a normalized solution under some suitable conditions on $V_i(x),\ i=1,2,3$.
	
	Recently, many authors studied the normalized solution of Schr{\"o}dinger systems, for example, Bartsch, Jeanjean, Soave \cite{bartsch2016normalized} and Bartsch, Soave \cite{bartsch2017natural} considered the following system
	\begin{equation}\label{bartsch}
		\begin{cases}
		-\Delta u+\lambda_{1} u=\mu_{1} u^{3}+\beta u v^{2}, & \text { in } \mathbb{R}^{N}, \\ -\Delta v+\lambda_{2} v=\mu_{2} v^{3}+\beta u^{2} v, & \text { in } \mathbb{R}^{N}, \\ \int_{\mathbb{R}^{N}} u^{2}=a, \int_{\mathbb{R}^{N}} v^{2}=b, & 
		\end{cases}
	\end{equation}
	with $\beta>0$ and $\beta<0$ respectively, they found solutions to (\ref{bartsch}) for specific range of $\beta$ depending on $a, b$. Recently, Bartsch, Zhong and Zou \cite{bartsch2021normalized} studied (\ref{bartsch}) for $\beta>0$ belonging to much better ranges independent of the masses $a$ and $b$, they adopted a new approach based on the fixed point index in cones and the bifurcation theory. For more results about (\ref{bartsch}), we refer to \cite{bartsch2018normalized,bartsch2019multiple} and the references therein. When considering (\ref{bartsch}) with more general exponents of the following type:
	\begin{equation}\label{general}
	\begin{cases}-\Delta u+\lambda_{1} u=\mu_{1}|u|^{p-2} u+\beta r_{1}|u|^{r_{1}-2}|v|^{r_{2}} u, & \text { in } \mathbb{R}^{N}, \\ -\Delta v+\lambda_{2} v=\mu_{2}|v|^{q-2} v+\beta r_{2}|u|^{r_{1}}|v|^{r_{2}-2} v, & \text { in } \mathbb{R}^{N}, \\ \int_{\mathbb{R}^{N}} u^{2}=a, \int_{\mathbb{R}^{N}} v^{2}=b, & \end{cases}
	\end{equation}
	we note that Gou and Jeanjean \cite{gou2016existence} proved the precompactness of the minimizing sequences up to translation for mass-subcritical problems. Yun-Zhang \cite{yun2022normalized} considered the system (\ref{general}) with linear couplings. Liu-Tian \cite{liu2020normalized} extended (\ref{general}) to the system of three equations. For more works, see, e.g.\cite{gou2018multiple}.
	
	The Schr{\"o}dinger systems like (\ref{key1}) with quadratic interaction have wide applications in physics, such as nonlinear optics, Bose-Einstein condensates and plasma physics. 
	The system
	\begin{equation}\label{sanyuan}
		\left\{ {\begin{array}{*{20}{c}}
				{i{\partial _t}{v_1} =  - \Delta {v_1} - {\mu _1}{{\left| {{v_1}} \right|}^{p - 2}}{v_1} - \beta {v_2}{v_3}}, \\ 
				{i{\partial _t}{v_2} =  - \Delta {v_2} - {\mu _2}{{\left| {{v_2}} \right|}^{p - 2}}{v_2} - \beta {v_1}{v_3}}, \\ 
				{i{\partial _t}{v_3} =  - \Delta {v_3} - {\mu _3}{{\left| {{v_3}} \right|}^{p - 2}}{v_3} - \beta {v_1}{v_2}}, 
		\end{array}} \right.
	\end{equation}
	and system (\ref{key1}) are reduced systems studied in \cite{colin2009instability,colin2006stability,colin2012bifurcation,pomponio2010ground}. These are unstable phenomena that occur when the incident laser field propagates into the plasma. The model was first introduced in document \cite{russell1999nonlinear} to describe Raman scattering in plasma. In reference \cite{colin2004quasilinear}, a modified model (\ref{sanyuan}) was derived to describe the nonlinear interaction between laser beam and plasma. From the physical point of view, when the incident laser field enters the plasma, it will be backscattered by the Raman type process. These two waves interact to generate electron plasma waves. These three waves combine to produce changes in ion density, which itself has an impact on the three waves. The system describing this phenomenon consists of three Schr{\"o}dinger equations coupled to the wave equation and is read in an appropriate dimensionless form. For a complete description of the model and an accurate description of the physical coefficients, we refer to \cite{colin2004quasilinear,colin2006numerical,colin2009instability,colin2006stability}. Systems similar to (\ref{sanyuan}) also appear as optical models with quadratic nonlinearity (see \cite{yew2000stability,yew2001multipulses}). 
	
	When $V_i(x)  \equiv 0,i=1,2,3$, the system (\ref{key1}) is reduced to 
	\begin{equation}\label{chang}
		\begin{cases} - \Delta u + {\lambda _1}u = {\mu _1}{{\left| u \right|}^{p - 2}}u + \beta vw,& \text { in }\ {\mathbb{R}^3}, \\ 
		 - \Delta v + {\lambda _2}v = {\mu _2}{{\left| v \right|}^{p - 2}}v + \beta uw,& \text { in }\ {\mathbb{R}^3},\\ 
		 - \Delta w + {\lambda _3}w = {\mu _3}{{\left| w \right|}^{p - 2}}w + \beta uv,&\text { in }\ {\mathbb{R}^3}, 
			\end{cases}
	\end{equation}
	which was studied by many authors. For example, Colin and his collaborators studied the orbital stability, ground state solution and multiple solutions in \cite{colin2009instability,colin2006stability,colin2012bifurcation}. In Wang \cite{wang2017solitary}, he studied the positive ground solution, bifurcation phenomenon, and asymptotic behavior of solutions under some suitable conditions on $\lambda$, $\mu$ and $\beta$. When two of $u,v,w$ are equal, Wang \cite{wang2022existence} studied the normal solution of (\ref{chang}).
	
	When $V_i(x) \not \equiv 0$, under some other conditions on $V_i(x)$, there are many results of normalized solutions of Schr{\"o}dinger systems. When the system has two equations, Chen-Zou \cite{chen2021normalized} which inspired the result of our article, considered a system with linear coupling terms. They got that when $V_i(x),i=1,2$ satisfy $(V1)(V2)$ or $(V1)(V3)$, the system
	$$
		\begin{cases}
				 - \Delta u + ({V_1}(x) + {\lambda _1})u = {\mu _1}{{\left| u \right|}^{p - 2}}u + \beta v,&\text{in}\  {\mathbb{R}^N}, \\ 
				 - \Delta v + ({V_2}(x) + {\lambda _2})v = {\mu _2}{{\left| v \right|}^{q - 2}}v + \beta u,&\text{in}\ {\mathbb{R}^N}, 
		\end{cases}
	$$
	has a ground state solution with the constraint 
	$$
		\int_{{\mathbb{R}^N}} {{u^2}}  = a,\int_{{\mathbb{R}^N}} {{v^2}}  = b,
	$$
	where 
	\begin{itemize}
		\item[$(V1)$] $V_i(x) \in C(\mathbb{R}^3 ) $, $V_i(0) =\mathop {\min }\limits_{x \in {\mathbb{R}^3}} {V_i}(x) = {c_i} >  - \infty $ for $i =1, 2,3,$ 
		\item[$(V2)$] $\mathop {\lim }\limits_{\left| x \right| \to \infty } {V_i}(x) =  + \infty $ for $i =1, 2,3.$
		\item[$(V3)$] $\mathop {\lim }\limits_{\left| x \right| \to \infty } {V_i}(x) = \mathop {\sup }\limits_{x \in {\mathbb{R}^N}} {V_i}(x) = :{V_{i,\infty }} \in \left( {{c_i}, + \infty } \right)$, for $i = 1,2.$
	\end{itemize}
	When considering the case where the system has three equations, the authors in \cite{kurata2021asymptotic,pomponio2010ground} studied the existence of ground state solution. They all assumed  
	$$
	0 < {C_i} < {V_i}(x) \leq  \mathop {\lim }\limits_{\left| y \right| \to \infty } {V_i}(y) = :{V_{i,\infty }} \in \mathbb{R},i=1,2,3,
	$$ 
	for almost every $x \in \mathbb{R}^N$. In \cite{ikoma2020compactness,kurata2022variational,osada2022existence}, the authors studied the existence and the orbital stability of minimizers and the precompactness of minimizing sequences under multiconstraint conditions for the problem (\ref{key1}). They all assumed that $V_i (x)$ is non-negative or non-positive. In \cite{kurata2022variational}, the authors studied the system
	\begin{equation}
	\left\{ {\begin{array}{*{20}{c}}
			{ - \Delta {u_1} + {V_1}(x){u_1} = {\mu _1}{{\left| {{u_1}} \right|}^{p - 2}}{u_1} + \beta {u_2}\overline {{u_3}} ,\quad \text { in }\ {\mathbb{R}^N},} \\ 
			{ - \Delta {u_2} + {V_2}(x){u_2} = {\mu _2}{{\left| {{u_2}} \right|}^{p - 2}}{u_2} + \beta {u_1}\overline {{u_3}} ,\quad \text { in }\ {\mathbb{R}^N},} \\ 
			{ - \Delta {u_3} + {V_3}(x){u_3} = {\mu _3}{{\left| {{u_3}} \right|}^{p - 2}}{u_3} + \beta {u_1}{u_2},\quad \text { in }\ {\mathbb{R}^N},} \\ 
			{{u_1},{u_2},{u_3} \in {H^1}({\mathbb{R}^N};\mathbb{C}),\left| {{u_1}} \right|_2^2 = \gamma ,\left| {{u_2}} \right|_2^2 = \mu ,\left| {{u_3}} \right|_2^2 = s,} 
	\end{array}} \right.
	\end{equation}
	with
	\begin{itemize}
		\item[$(\widetilde{V1})$]  $V \in L^\infty (\mathbb{R}^N; \mathbb{R})$; 
		\item[$(\widetilde{V2})$] $V (x) \leq \mathop {\lim }\limits_{\left| y \right| \to \infty } V(y) = 0$ for almost every $x \in \mathbb{R}^N $; 
		\item[$(\widetilde{V3})$] $V(-x_1, x') = V (x_1, x')$ for almost every $x_1 \in \mathbb{R}$ and $x' \in \mathbb{R}^{N-1}$, $V(s, x') \leq V(t, x')$ for almost every $s, t \in \mathbb{R}$ with $0 \leq s < t$ and $x' \in \mathbb{R}^{N-1}$.
	\end{itemize}	
	They defined
	\begin{equation*}
		\begin{aligned}
			E(\vec{u}):=&\frac{1}{2} \sum_{j=1}^{3} \int_{\mathbb{R}^{N}}\left|\nabla u_{j}\right|^{2} +\frac{1}{2} \sum_{j=1}^{3} \int_{\mathbb{R}^{N}} V_{j}(x)\left|u_{j}\right|^{2}  \\&
			-\frac{1}{p+1} \sum_{j=1}^{3} \int_{\mathbb{R}^{N}}\left|u_{j}\right|^{p+1} -\alpha \operatorname{Re} \int_{\mathbb{R}^{N}} u_{1} u_{2} \bar{u}_{3} 
		\end{aligned}
	\end{equation*} 
	where $\vec{u}=(u_1,u_2,u_3)$.
	And they let 
	\begin{equation*}
		I(\gamma ,\mu ,\operatorname{s} ): = \inf \left\{ {E(u_1,u_2,u_3 ):(u_1,u_2,u_3)  \in {H^1}({\mathbb{R}^N};{\mathbb{C}^3}),\left| u_1 \right|_2^2 = \gamma ,\left| u_2 \right|_2^2 = \mu ,\left| u_3 \right|_2^2 = s} \right\}.
	\end{equation*}
	They got the following proposition.
	\begin{proposition}
		Suppose $\gamma, \mu, s>0, N=1,2,3$, $1<p<1+4 / N, \beta>0$ and $V_{j}(x)(j=1,2,3)$ satisfy the conditions $(\widetilde{V1})$-$(\widetilde{V3})$, respectively. Then, any minimizing sequence $\left\{\vec{u}_{n}\right\}_{n=1}^{\infty}$ for $I(\gamma, \mu, s)$ is relatively compact in $H^{1}\left(\mathbb{R}^{N} ; \mathbb{C}^{3}\right)$ up to translations. That is, there exist $\left\{y_{n}\right\}_{n=1}^{\infty} \subset \mathbb{R}^{N}$ and $\vec{u} \in H^{1}\left(\mathbb{R}^{N} ; \mathbb{C}^{3}\right)$ such that $\left\{\vec{u}_{n}\left(\cdot+y_{n}\right)\right\}_{n=1}^{\infty}$ has a subsequence converging strongly in $H^{1}\left(\mathbb{R}^{N} ; \mathbb{C}^{3}\right)$ to $\vec{u}$. Moreover, $\vec{u}$ is a minimizer for $I(\gamma, \mu, s)$.
			
	\end{proposition}
	Especially, when $V_i(x) \equiv 0$, system (\ref{chang}) with the restriction of (\ref{key2}) has a solution. 
	
	Let $H^1 (\mathbb{R}^3)$ be the usual Sobolev space and denote its norm by
	\begin{equation*}
		\|u\|:=(|\nabla u|^2_2+|u|^2_2)^{1/2},
	\end{equation*} 
	where $|u|_q:=|u|_{q,\mathbb{R}^3}:=(\int_{{\mathbb{R}^3}} |u|^q)^{1/q}$, $q>1$ and we know that $H^1(\mathbb{R}^3) \subset L^q(\mathbb{R}^3)$, $2 \leq q \leq 6$.
	Let $H=H^1(\mathbb{R}^3)  \times H^1(\mathbb{R}^3) \times H^1 (\mathbb{R}^3)$	and we set 
	\begin{equation*}
		S(a,b,c ): = \left\{ {(u,v,w) \in \mathcal{H}:|u|_2^2 = {a},|v|_2^2 = {b},|w|_2^2 = {c}} \right\}, 
	\end{equation*}
	where 
	$$
	\mathcal{H}: = \left\{ {\left( {u,v,w} \right) \in H:\left| {\int_{{\mathbb{R}^3}} {{V_1}(x){u^2}} } \right| < \infty ,\left| {\int_{{\mathbb{R}^3}} {{V_2}(x){v^2}} } \right| < \infty ,\left| {\int_{{\mathbb{R}^3}} {{V_3}(x){w^2}} } \right| < \infty } \right\}.
	$$
	We defined 
	\begin{equation}\label{min}
		m(a,b,c ): = \mathop {\inf }\limits_{(u,v,w) \in S(a,b,c) } J(u,v,w).
	\end{equation}
	From standard variational arguments we know that critical points of the following functional on $S(a,b,c) $ are weak solutions of (\ref{key1})-(\ref{key2}),
	\begin{equation*}
		\begin{aligned}
			J(u,v,w) =& \frac{1}{2}\int_{{\mathbb{R}^3}} {\left( {{{\left| {\nabla u} \right|}^2} + {V_1}{u^2}} \right)}  + \frac{1}{2}\int_{{\mathbb{R}^3}} {\left( {{{\left| {\nabla v} \right|}^2} + {V_2}{v^2}} \right)}  + \frac{1}{2}\int_{{\mathbb{R}^3}} {\left( {{{\left| {\nabla w} \right|}^2} + {V_3}{w^2}} \right)} \\&- \frac{{{\mu _1}}}{p}\left| u \right|_p^p - \frac{{{\mu _2}}}{p}\left| v \right|_p^p - \frac{{{\mu _2}}}{p}\left| w \right|_p^p - \beta\int_{{\mathbb{R}^3}} {uvw} .
		\end{aligned}
	\end{equation*}
	For comparison, we define
	\begin{equation*}
		\begin{aligned}
			{J_\infty }(u,v,w) =& \frac{1}{2}\int_{{\mathbb{R}^3}} {\left( {{{\left| {\nabla u} \right|}^2} + {{\left| {\nabla v} \right|}^2} + {{\left| {\nabla w} \right|}^2}} \right)} \\& - \frac{{{\mu _1}}}{p}\left| u \right|_p^p - \frac{{{\mu _2}}}{p}\left| v \right|_p^p - \frac{{{\mu _3}}}{p}\left| w \right|_p^p - \beta\int_{{\mathbb{R}^3}} {uvw} ,
		\end{aligned}
	\end{equation*}
	\begin{equation}\label{defination_of_S_infty}
		S_\infty(a,b,c)=\left\{ {(u,v,w) \in H:|u|_2^2 = {a},|v|_2^2 = {b},|w|_2^2 = {c} } \right\},
	\end{equation}
	and
	\begin{equation*}
		m_\infty (a,b,c):=\mathop {\inf }\limits_{(u,v,w)\in S_\infty(a,b,c)} {J_\infty }(u,v,w).
	\end{equation*}
	It is easy to see $S(a,b,c) \subset S_\infty (a,b,c)$.
	
	The following theorem is the main result of our paper.
	\begin{theorem}\label{th2}
		Let $\beta>0$, $\mu_{1},\mu_{2},\mu_3>0 $, $2<p<10/3$. Suppose $V_i(x),i=1,2,3$ satisfy $(V1)$ and $(V2)$, Then $m(a,b,c)$ is attained, hence there exist a ground state solution to $(\ref{key1})$ with the constraint $\left( {\overline u ,\overline v ,\overline w} \right) \in S(a,b,c)$. Moreover $\overline{u}, \overline{v}, \overline{w} >0$.
	\end{theorem} 
	
	Note that we do not assume $c_i>0$ in $(V1)$, therefore the space we use has no compactness although $V(x)$ is coercive. The loss of compactness caused major difficulties, and we restored compactness through a series of variational estimates.
	In Section 2, we present necessary notations and give the some preparations on problem (\ref{key1})-(\ref{key2}). The proof of Theorem \ref{th2} will be given in Section 3. Throughout the paper we use the notation $\left|  \cdot  \right|_p$ to denote the $L^p(\mathbb{R}^3)$, and we simply write $H^1=H^1(\mathbb{R}^3),H=H^1(\mathbb{R}^3) \times H^1(\mathbb{R}^3) \times H^1(\mathbb{R}^3)$. The symbol $\left\|  \cdot  \right\|$ denotes the norm in $H^1$ or $H$. The notation $\rightharpoonup$ denotes weak convergence in $H^1$ or $H$. Capital latter $C$ stands for positive constant which may depend on $p, q$, whose precise value can change from line to line. 

	\section{Some preparations} 
	In this section, we prove some lemmas which play important roles in the proof of our main result. We work on the space $H$, and we need the following basic Gagliardo-Nirenberg's inequality from \cite{adams2003sobolev}.
	\begin{lemma}[Gagliardo-Nirenberg inequality]\label{GN}
		For any $u \in H^1(\mathbb{R}^3)$ and $q\in [2,6)$, we have
		\begin{equation}
			|u|_q \leq C_q|\nabla u|^{\gamma_{q} } _2|u|^{1-\gamma_{q}}_2,
		\end{equation} 
		where $\gamma_{q}=\frac{3(q-2)}{2q}$.
	\end{lemma}
	\begin{lemma}\label{coercive}
		$J(u,v,w)$ is coercive and bounded from below on $S(a,b,c)$.
	\end{lemma}
	\begin{proof}
		By Lemma \ref{GN} and $2<p<10/3$ we have 
		\begin{equation}\label{p-GN}
			\int_{{\mathbb{R}^3}} {{{\left| u \right|}^{{p}}} \leq C({p},{a})\left| {\nabla u} \right|_2^{\frac{{3(p - 2)}}{2}}}.
		\end{equation}
		These are also similar to $v$ and $w$.
		
		By Young inequality and (\ref{p-GN}) we have
		\begin{equation*}
			\int_{{\mathbb{R}^3}} {uvw}  \leq \frac{1}{3}\int_{{\mathbb{R}^3}} {{{\left| u \right|}^3}}  + \frac{1}{3}\int_{{\mathbb{R}^3}} {{{\left| v \right|}^3}}  + \frac{1}{3}\int_{{\mathbb{R}^3}} {{{\left| w \right|}^3}}  \leq C\left( {\left| {\nabla u} \right|_2^{\frac{3}{2}} + \left| {\nabla v} \right|_2^{\frac{3}{2}} + \left| {\nabla w} \right|_2^{\frac{3}{2}}} \right).
		\end{equation*}
		By $(V1)$ we have
		$$
			\int_{{\mathbb{R}^3}} {{V_1}(x){u^2} \geq  {c_1}a} ,\int_{{\mathbb{R}^3}} {{V_2}(x){v^2} \geq  {c_2}b} ,\int_{{\mathbb{R}^3}} {{V_3}(x){w^2} \geq  {c_3}c} ,
		$$
		where $c_i,i=1,2,3$ were defined in $(V1)$.
		Thus for any $(u,v,w)\in S(a,b,c)  $,
		\begin{equation}\label{qiangzhi}
			\begin{aligned}
				J(u,v,w) \geq & \frac{1}{2}\left( {\int_{{\mathbb{R}^3}} {{{\left| {\nabla u} \right|}^2}}  + \int_{{\mathbb{R}^3}} {{{\left| {\nabla v} \right|}^2}}  + \int_{{\mathbb{R}^3}} {{{\left| {\nabla w} \right|}^2}} } \right)  \\&
				- \frac{{{\mu _1}C({p},{a})}}{{{p}}}\left| {\nabla u} \right|_2^{\frac{{3({p} - 2)}}{2}} - \frac{{{\mu _2}C({p},{b})}}{{{p}}}\left| {\nabla v} \right|_2^{\frac{{3({p} - 2)}}{2}} - \frac{{{\mu _3}C({p},{c})}}{{{p}}}\left| {\nabla w} \right|_2^{\frac{{3({p} - 2)}}{2}}  \\&
				- C\left( {\left| {\nabla u} \right|_2^{\frac{3}{2}} + \left| {\nabla v} \right|_2^{\frac{3}{2}} + \left| {\nabla w} \right|_2^{\frac{3}{2}}} \right)-C(c_1,c_2,c_3,a,b,c).
			\end{aligned}
		\end{equation}
		Since $\left| u \right|_2^2,\left| v \right|_2^2,\left| w \right|_2^2$ are fixed in $S(a,b,c)$, combing with $p < \frac{10}{3}$ we know $J(u,v,w)$ is coercive and bounded from below on $S(a,b,c) $. 
	\end{proof}
	\begin{lemma}\label{m_infty}
		$m_\infty(a,b,c) >-\infty.$
	\end{lemma}
	\begin{proof}
		The proof is similar to that of Lemma \ref{coercive}.
	\end{proof}
	In the proof of Theorems \ref{th2}, the most important step is to verify the convergence of minimizing sequences. Here we recall the following lemma without proof, due to Brezis and Lieb \cite{brezis1983relation}.
	\begin{lemma}\label{error}
		Suppose $\left\{\left(u_n, v_n, w_n\right)\right\} \subset H$ is a bounded sequence, $\left(u_n, v_n, w_n\right) \rightharpoonup (u, v, w)$ in $H$, then for $1<s<\infty$, we have
		$$
		\begin{gathered}
			\lim _{n \rightarrow \infty} \int_{ \mathbb{R}^{3}}\left|\nabla u_{n}\right|^{2}-|\nabla u|^{2}-\left|\nabla\left(u_{n}-u\right)\right|^{2}  =0, \lim _{n \rightarrow \infty} \int_{ \mathbb{R}^{3}}\left|\nabla v_{n}\right|^{2}-|\nabla v|^{2}-\left|\nabla\left(v_{n}-v\right)\right|^{2}  =0 , \\
			\lim _{n \rightarrow \infty} \int_{ \mathbb{R}^{3}}\left|\nabla w_{n}\right|^{2}-|\nabla w|^{2}-\left|\nabla\left(w_{n}-w\right)\right|^{2}  =0 ,\\
			\lim _{n \rightarrow \infty} \int_{ \mathbb{R}^{3}}\left|u_{n}\right|^{s}-|u|^{s}-\left|u_{n}-u\right|^{s}  =0, \lim _{n \rightarrow \infty} \int_{ \mathbb{R}^{3}}\left|v_{n}\right|^{s}-|v|^{s}-\left|v_{n}-v\right|^{s}  =0 , \\		
			\lim _{n \rightarrow \infty} \int_{ \mathbb{R}^{3}}\left|w_{n}\right|^{s}-|w|^{s}-\left|w_{n}-w\right|^{s}  =0.	
		\end{gathered}
		$$
	\end{lemma} 	
	\begin{lemma}\label{B-L}
		Assume $(u_n,v_n,w_n) \rightharpoonup (u,v,w)$ in $H$, then
		\begin{equation}
			\mathop {\lim }\limits_{n \to \infty } \int_{{{\mathbb{R}}^3}} {{u_n}} {v_n}{w_n} - uvw - \left( {{u_n} - u} \right)\left( {{v_n} - v} \right)\left( {{w_n} - w} \right)  = 0. 
		\end{equation}
	\end{lemma}
\begin{proof}
	We follow the idea of Lemma 2.3 of \cite{chen2015existence} and Lemma 3.2 of \cite{liang2020existence}. For any $b_1,b_2,c_1,c_2,d_1,d_2 \in \mathbb{R}$ and $\epsilon>0$, the Young's inequality leads to 
	\begin{equation}\label{chu}
		\begin{aligned}
					\left| {({b_1} + {b_2})({c_1} + {c_2})({d_1} + {d_2}) - {b_1}{c_1}{d_1}} \right| \leq& \left| {{b_1}{c_2}{d_1}} \right| + \left| {{b_2}{c_1}{d_1}} \right| + \left| {{b_2}{c_2}{d_1}} \right| + \left| {{b_1}{c_1}{d_2}} \right| \hfill \\&
					+ \left| {{b_1}{c_2}{d_2}} \right| + \left| {{b_2}{c_2}{d_1}} \right| + \left| {{b_2}{c_2}{d_2}} \right| \hfill \\
					\leq & \varepsilon \left( {{{\left| {{b_1}} \right|}^3} + {{\left| {{c_1}} \right|}^3} + {{\left| {{d_1}} \right|}^3}} \right) + \frac{\varepsilon }{3}\left( {2{{\left| {{b_2}} \right|}^3} + 2{{\left| {{c_2}} \right|}^3} + {{\left| {{d_2}} \right|}^3}} \right)\\& + \frac{{{\varepsilon ^{ - 2}}}}{3}\left( {2{{\left| {{b_2}} \right|}^3} + 2{{\left| {{c_2}} \right|}^3} + {3{\left| {{d_2}} \right|}^3}} \right) \hfill \\ 
		\end{aligned} 
	\end{equation}
	Letting $b_1:=u_n-u,b_2=u$; $c_1:=v_n-v,c_2=v$; $d_1:=w_n-w,d_2=w$ and substituting them into the estimate (\ref{chu}), we have
	\begin{equation}
	\begin{aligned}
			f_n^\varepsilon : = &[\left| {{u_n}{v_n}{w_n} - ({u_n} - u)({v_n} - v)({w_n} - w) - uvw} \right| \hfill \\
			\\&- \varepsilon \left( {{{\left| {{u_n} - u} \right|}^3} + {{\left| {{v_n} - v} \right|}^3} + {{\left| {{w_n} - w} \right|}^3}} \right) - \frac{\varepsilon }{3}\left( {2{{\left| u \right|}^3} + 2{{\left| v \right|}^3} + {{\left| w \right|}^3}} \right){]^ + }  \\ 
			\leq & \left| {uvw} \right| + \frac{{{\varepsilon ^{ - 2}}}}{3}\left( {2{{\left| u \right|}^3} + 2{{\left| v \right|}^3} + {3{\left| w \right|}^3}} \right), \hfill \\ 
	\end{aligned} 
	\end{equation}
	where $u^+(x):=\max \{u(x),0\}$. Then the dominated convergence theorem implies that 
	\begin{equation}\label{o(n)}
	\int_{{\mathbb{R}^3}} {f_n^\varepsilon  \to 0\ ,as\ n \to \infty } .
	\end{equation}
	Since 
	\begin{equation*}
		\begin{aligned}
			&\left| {{u_n}{v_n}{w_n} - ({u_n} - u)({v_n} - v)({w_n} - w) - uvw} \right| \\ \leq&  f_n^\varepsilon  + \varepsilon \left( {{{\left| {{u_1} - u} \right|}^3} + {{\left| {{v_1} - v} \right|}^3} + {{\left| {{w_1} - w} \right|}^3}} \right)
			+\frac{\varepsilon }{3}\left( {2{{\left| u \right|}^3} + 2{{\left| v \right|}^3} + {{\left| w \right|}^3}} \right),
		\end{aligned}
	\end{equation*}
	the boundedness of ${(u_n,v_n,w_n)}$ in $H$, the formula (\ref{o(n)}) and the arbitrariness of $\epsilon$ deduce that 
	$$
	\int_{{\mathbb{R}^3}} {({u_n}{v_n}{w_n} - ({u_n} - u)({v_n} - v)({w_n} - w)) = \int_{{\mathbb{R}^3}} {uvw}  + o_n(1)} .
	$$
	The proof is complete.
\end{proof}
	If we want the solution of (\ref{key1})-(\ref{key2}) to be positive, we need the following lemma.
	\begin{lemma}\label{feifulie}
		Let $\{(u_n, v_n, w_n)\} \subset S(a,b,c) $ be a minimizing sequence for $m(a,b,c)$. Then for $\beta > 0$, $\{(|u_n|, |v_n|, |w_n|)\}$ is also a minimizing sequence.
	\end{lemma}
	\begin{proof}
		Since
		\begin{equation*}
			\int_{\mathbb{R}^{3}}|\nabla| u_{n} \|^{2}  \leq \int_{\mathbb{R}^3}\left|\nabla u_{n}\right|^{2} , \int_{\mathbb{R}^{3}}|\nabla| v_{n} \|^{2}  \leq \int_{\mathbb{R}^3}\left|\nabla v_{n}\right|^{2} , 
			\int_{\mathbb{R}^{3}}|\nabla| w_{n} \|^{2}  \leq \int_{\mathbb{R}^3}\left|\nabla w_{n}\right|^{2} , 
		\end{equation*}
		\begin{equation*}
			\int_{\mathbb{R}^3} u_{n} v_{n} w_n  \leq \int_{\mathbb{R}^3}\left|u_{n}\right| \left|v_{n}\right|  \left|w_{n}\right|,
		\end{equation*}
		we have
		$$
		\begin{aligned}
			{J }\left( {\left| {{u_n}} \right|,\left| {{v_n}} \right|,\left| {{w_n}} \right|} \right) =& \frac{1}{2}\int_{{\mathbb{R}^3}} {\left( {{{\left| {\nabla \left| u \right|} \right|}^2} + {V_1}(x){u^2}} \right)}  + \frac{1}{2}\int_{{\mathbb{R}^3}} {\left( {{{\left| {\nabla \left| v \right|} \right|}^2} + {V_2}(x){v^2}} \right)} \\ & + 
			\frac{1}{2}\int_{{\mathbb{R}^3}} {\left( {{{\left| {\nabla \left| w \right|} \right|}^2} + {V_3}(x){w^2}} \right)}    - \frac{1}{{{p}}}\int_{{{\mathbb R}^3}} {{\mu _1}{{\left| {{u_n}} \right|}^{{p}}}}  - \frac{1}{{{p}}}\int_{{{\mathbb R}^3}} {{\mu _2}{{\left| {{v_n}} \right|}^{{p}}}}  \\&- \frac{1}{{{p}}}\int_{{{\mathbb R}^3}} {{\mu _3}{{\left| {{w_n}} \right|}^{{p}}}}  - \beta \int_{{{\mathbb R}^3}} {\left| {{u_n}} \right|} \left| {{v_n}} \right|\left| {{w_n}} \right|  \\
			\leq &   \frac{1}{2}\int_{{\mathbb{R}^3}} {\left( {{{\left| {\nabla u} \right|}^2} + {V_1}(x){u^2}} \right)}  + \frac{1}{2}\int_{{\mathbb{R}^3}} {\left( {{{\left| {\nabla v} \right|}^2} + {V_2}(x){v^2}} \right)}\\&   + \frac{1}{2}\int_{{\mathbb{R}^3}} {\left( {{{\left| {\nabla w} \right|}^2} + {V_3}(x){w^2}} \right)}    - \frac{1}{{{p}}}\int_{{{\mathbb R}^3}} {{\mu _1}{{\left| {{u_n}} \right|}^{{p}}}}  - \frac{1}{{{p}}}\int_{{{\mathbb R}^3}} {{\mu _2}{{\left| {{v_n}} \right|}^{{p}}}}  \\&- \frac{1}{{{p}}}\int_{{{\mathbb R}^3}} {{\mu _3}{{\left| {{w_n}} \right|}^{{p}}}}  - \beta \int_{\mathbb {R}^3} {{u_n}} {v_n}{w_n}  \\
			=&  {J}\left( {{u_n},{v_n},{w_n}} \right).
		\end{aligned}
		$$
		Thus, $J\left(\left|u_{n}\right|,\left|v_{n}\right|,\left|w_{n}\right| \right) \leq J\left(u_{n}, v_{n},w_n \right)$. Therefore, when $\beta>0$, $\left\{\left(\left|u_{n}\right|,\left|v_{n}\right|,\left|w_{n}\right|\right)\right\}$ is also a minimizing sequence.
	\end{proof}
	
	\section{Proof of Theorem \ref{th2}} 
	The following two lemmas are critical for the proof of theorem \ref{th2}.
	\begin{lemma}\label{wei_shi_jing_xing}
		Suppose $\mu_i>0,i=1,2,3,\beta >0,2<p<10/3$. Assume $V_{i}(x),i=1,2,3$ satisfy $(V1)$ and $(V2)$. Then any minimizing sequence $\left\{\left(u_n, v_n, w_n\right)\right\} \subset S(a,b,c )$ with $J\left(u_n, v_n, w_n \right) \rightarrow m(a,b,c )$, possesses a convergent subsequence in $L^{2}\left(\mathbb{R}^{3}\right) \times L^{2}\left(\mathbb{R}^{3}\right) \times L^{2}\left(\mathbb{R}^{3}\right)$.
	\end{lemma}
	\begin{proof} 
		By the coerciveness of $J$ on $S(a,b,c )$, the sequence $\left\{\left(u_n, v_n, w_n \right)\right\}$ is bounded in $\mathcal{H}$. So we may assume that $\left(u_n, v_n, w_n \right) \rightharpoonup(\bar{u}, \bar{v}, \bar{u})$ in $\mathcal{H}$ along a subsequence, and also holds in $H$. 
		Now we verify the compactness of $\left(u_n, v_n, w_n \right)$. 
		Put $\left( {\widetilde{u}_n,\widetilde{v}_n,\widetilde{w}_n } \right) = ({u_n} - \bar u,{v_n} - \bar v,{w_n} - \bar w)$,
		then we have $\left(\widetilde{u}_{n}, \widetilde{v}_{n}, \widetilde{w}_{n}\right) \rightharpoonup 0$ in $H$. 
		We argue by contradiction that
		$$
		\delta:=\liminf _{n \rightarrow \infty} \int_{\mathbb{R}^{3}}\left(\tilde{u}_{n}^{2}+\tilde{v}_{n}^{2}+\tilde{w}_{n}^{2}\right) >0.
		$$
		
		Up to a subsequence, we may assume that $\left(\tilde{u}_{n}, \tilde{v}_{n},\tilde{w}_{n}\right) \rightarrow 0$ in $L_{l o c}^{2}\left(\mathbb{R}^{3}\right) \times L_{\text {loc }}^{2}\left(\mathbb{R}^{3}\right) \times L_{\text {loc }}^{2}\left(\mathbb{R}^{3}\right)$. Then by $(V2)$ we have that
		$$
		\int_{\mathbb{R}^{3}}\left(V_{1}(x) \tilde{u}_{\mathrm{n}}^{2}+V_{2}(x) \tilde{v}_{\mathrm{n}}^{2}+V_{3}(x) \tilde{w}_{\mathrm{n}}^{2}\right)  \rightarrow \infty,
		$$
		which implies that
		$$
		\int_{\mathbb{R}^{3}}\left(V_{1}(x) u_{n}^{2}+V_{2}(x) v_{n}^{2}+V_{3}(x) w_{n}^{2}\right)  \rightarrow \infty.
		$$
		Recalling that $\left\{\left(u_n, v_n, w_n \right)\right\}$ is bounded in $H$ and $(u_n,v_n,w_n) \in S(a,b,c) \subset S_\infty(a,b,c)$, by Lemma \ref{m_infty} we obtain that
		$$
		\begin{aligned}
			m(a,b,c ) &=J\left(u_n, v_n, w_n \right)+o(1) \\
			&=J_{\infty}\left(u_{n}, v_{n},w_n \right)+\frac{1}{2} \int_{\mathbb{R}^{3}} \left( V_{1}(x) u_{n}^{2}+V_{2}(x) v_{n}^{2}+ V_{3}(x) w_{n}^{2}\right)+o(1)  \\
			& \geq m_{\infty}(a,b,c )+\frac{1}{2} \int_{\mathbb{R}^{3}} \left( V_{1}(x) u_{n}^{2}+V_{2}(x) v_{n}^{2} +V_{3}(x) w_{n}^{2}\right) \\
			& \rightarrow+\infty,
		\end{aligned}
		$$
		which is a contradiction. So we obtain that $\left(u_n, v_n, w_n \right) \rightarrow(\bar{u}, \bar{v},\bar{w})$ in $L^{2}\left(\mathbb{R}^{3}\right) \times L^{2}\left(\mathbb{R}^{3}\right) \times L^{2}\left(\mathbb{R}^{3}\right)$. 
	\end{proof}
	\begin{lemma}\label{strong_convergence}
		If $\left(u_{0}, v_{0},w_0\right) \in H$ and $\left\{\left(u_{n}, v_{n},w_n \right)\right\}$ is a minimizing sequence for $m(a,b,c )$ with $\left(u_{n}, v_{n},w_n \right) \rightarrow\left(u_{0}, v_{0},w_0 \right)$ in $L^{p}\left(\mathbb{R}^{3}\right) \times L^{q}\left(\mathbb{R}^{3}\right)\times L^{r}\left(\mathbb{R}^{3}\right)$ for $2 \leq  p, q,r<6$, then $\left(u_{n}, v_{n},w_n \right) \rightarrow\left(u_{0}, v_{0},w_0 \right)$ in $H$.
	\end{lemma}
	\begin{proof}
		Since $\left|u_{n}-u_{0}\right|_{2}^{2}+\left|v_{n}-v_{0}\right|_{2}^{2} +\left|w_{n}-w_{0}\right|_{2}^{2}\rightarrow 0$, we have that $\left|u_{0}\right|_{2}^{2}=a,\left|v_{0}\right|_{2}^{2}=b,\left|w_{0}\right|_{2}^{2}=c$.
		We claim that
		$$
		\lim _{n \rightarrow \infty} \int_{\mathbb{R}^{3}}\left(V_{1}(x) u_{n}^{2}+V_{2}(x) v_{n}^{2}+V_{3}(x) w_{n}^{2}\right) =\int_{\mathbb{R}^{3}}\left(V_{1}(x) u_{0}^{2}+V_{2}(x) v_{0}^{2}+(V_{3}(x) w_{0}^{2}\right)  .
		$$
		In fact, since $\mathop {\lim }\limits_{|x| \to \infty }V(x)=\infty$, there is a $R>0$ such that $V(x)\geq 0,\forall x\in \mathbb{R}^3\backslash B_R$ where $B_R:=\{x\in \mathbb{R}^3:|x|<R\} $. By Fatou lemma, 
		$$
		\int_{{\mathbb{R}^3}\backslash {B_R}} {{V_1}(x)u_0^2}  \leq \mathop {\lim }\limits_{n \to \infty } \int_{{\mathbb{R}^3}\backslash {B_R}} {{V_1}(x)u_n^2} .
		$$
		Since $V_i (x)  \in C(\mathbb{R}^3)$, then $0 \leq |V_i (x)| \leq M$, $x\in B_R$. There is ,
		$$
		0 \leftarrow  - M\int_{{B_R}} {\left| {u_n^2 - u_0^2} \right|}  \leq  - \int_{{B_R}} {\left| {{V_1}(x)} \right|\left| {u_n^2 - u_0^2} \right|}  \leq \int_{{B_R}} {{V_1}(x)(u_n^2 - u_0^2) \leq } M\int_{{B_R}} {\left| {u_n^2 - u_0^2} \right|}  \to 0.
		$$
		Therefore
		\begin{equation}\label{down}
			0 \leq \mathop {\lim }\limits_{n \to \infty } \int_{{\mathbb{R}^3}} {{V_1}(x)\left( {u_n^2 - u_0^2} \right)} .
		\end{equation}
		On the other hand, since $\{u_n\}$ is a minimizing sequence for $m(a,b)$, there is
		\begin{equation*}
			\begin{aligned}
				m(a,b) =& \mathop {\lim }\limits_{n \to \infty } J({u_n},{v_n},{w_n}) \hfill \\
				=& J({u_0},{v_0},{w_0}) + \mathop {\lim }\limits_{n \to \infty } \{ \frac{1}{2}\int_{{\mathbb{R}^3}} {\left( {{{\left| {\nabla {u_n}} \right|}^2} - {{\left| {\nabla {u_0}} \right|}^2}} \right)}  + \frac{1}{2}\int_{{\mathbb{R}^3}} {\left( {{{\left| {\nabla {v_n}} \right|}^2} - {{\left| {\nabla {v_0}} \right|}^2}} \right)} \\& + \frac{1}{2}\int_{{\mathbb{R}^3}} {\left( {{{\left| {\nabla {w_n}} \right|}^2} - {{\left| {\nabla {w_0}} \right|}^2}} \right)  }  + \frac{1}{2}\int_{{\mathbb{R}^3}} {\left( {{V_1}(x)u_n^2 - {V_1}(x)u_0^2} \right)}  + \frac{1}{2}\int_{{\mathbb{R}^3}} {\left( {{V_2}(x)v_n^2 - {V_2}(x)v_0^2} \right)}   \\& + \frac{1}{2}\int_{{\mathbb{R}^3}} {\left( {{V_3}(x)w_n^2 - {V_3}(x)w_0^2} \right)} \}   \hfill \\
				\geqslant & m(a,b) + \mathop {\lim \inf }\limits_{n \to \infty } \{ \frac{1}{2}\int_{{\mathbb{R}^3}} {\left( {{V_1}(x)u_n^2 - {V_1}(x)u_0^2} \right)}  + \frac{1}{2}\int_{{\mathbb{R}^3}} {\left( {{V_2}(x)v_n^2 - {V_2}(x)v_0^2} \right)}  \\&+ \frac{1}{2}\int_{{\mathbb{R}^3}} {\left( {{V_3}(x)w_n^2 - {V_3}(x)w_0^2} \right)} \}  , \hfill \\ 
			\end{aligned}
		\end{equation*}
		which implies 
		\begin{equation}\label{up}
			0 \geqslant \mathop {\lim \inf }\limits_{n \to \infty } \{ \frac{1}{2}\int_{{\mathbb{R}^3}} {{V_1}(x)\left( {u_n^2 - u_0^2} \right)}  + \frac{1}{2}\int_{{\mathbb{R}^3}} {{V_2}(x)\left( {v_n^2 - v_0^2} \right)}  + \frac{1}{2}\int_{{\mathbb{R}^3}} {{V_3}(x)\left( {w_n^2 - w_0^2} \right)} \} .
		\end{equation}
		Combing with (\ref{down}) and (\ref{up}), there is a subsequence still written as $\{u_n\}$ such that 
		$$
		\mathop {\lim }\limits_{n \to \infty } \int_{{\mathbb{R}^3}} {({V_1}(x)u_n^2 + {V_2}(x)v_n^2 + {V_3}(x)w_n^2)}  = \int_{{\mathbb{R}^3}} {({V_1}(x)u_0^2 + {V_2}(x)v_0^2 + {V_3}(x)w_0^2)} .
		$$
		Then the claim is proved.
		
		By the coerciveness of $J$ on $S(a,b,c )$, the sequence $\left\{\left(u_{n}, v_{n},w_n \right)\right\}$ is bounded in $H$.
		Up to a subsequence, suppose that $\left(u_{n}, v_{n},w_n \right) \rightharpoonup (\bar{u}, \bar{v},\bar{w}) \in H$. Thanks to $\left(u_{n}, v_{n},w_n \right) \rightarrow\left(u_{0}, v_{0},w_0 \right)$ in
		${L^2}\left( {{\mathbb{R}^3}} \right) \times {L^2}\left( {{\mathbb{R}^3}} \right) \times {L^2}\left( {{\mathbb{R}^3}} \right)$,
		we obtain that $(\bar{u}, \bar{v},\bar{w})=$ $\left(u_{0}, v_{0},w_0 \right)$, hence $\left(u_{n}, v_{n},w_n \right) \rightharpoonup \left(u_{0}, v_{0},w_0 \right)$ in $H$. Thus by Lemma \ref{error} and Lemma \ref{B-L}, we have that
		$$
		\begin{aligned}
			m(a,b,c ) & \leq  J\left(u_{0}, v_{0},w_0 \right) \\
			& \leq  \liminf _{n \rightarrow \infty} J\left(u_{n}, v_{n},w_n \right) \\
			&=m(a,b,c ),
		\end{aligned}
		$$
		which implies
		$$
		\lim _{n \rightarrow \infty} \int_{\mathbb{R}^{3}}\left(\left|\nabla u_{n}\right|^{2}+\left|\nabla v_{n}\right|^{2}+\left|\nabla w_{n}\right|^{2}\right) =\int_{{\mathbb{R}^3}} {\left( {{{\left| {\nabla {u_0}} \right|}^2} + {{\left| {\nabla {v_0}} \right|}^2} + {{\left| {\nabla {w_0}} \right|}^2}} \right)}  .
		$$
		
		Combining this fact with $\left(\nabla u_{n}, \nabla v_{n},\nabla w_{n}\right) \rightharpoonup \left(\nabla u_{0}, \nabla v_{0},\nabla w_{0}\right)$
		weakly in $L^{2}\left(\mathbb{R}^{3}\right) \times   L^{2}\left(\mathbb{R}^{3} \right) \times L^{2}\left(\mathbb{R}^{3}\right)$, 
		we observe that 
		$$
		\left| {\nabla \left( {{u_n} - {u_0}} \right)} \right|_2^2 + \left| {\nabla \left( {{v_n} - {v_0}} \right)} \right|_2^2 + \left| {\nabla \left( {{w_n} - {w_0}} \right)} \right|_2^2 \to 0,
		$$
		and thus
		$$
		\left(u_{n}, v_{n},w_n \right) \rightarrow\left(u_{0}, v_{0},w_0 \right) \ \text { in }\  H.
		$$
	\end{proof}
	\paragraph{Proof of Theorem \ref{th2}:} By Lemma \ref{coercive} we know $J(u,v,w)$ is bounded from below and coercive on $S(a,b,c) $, then we can find a bounded minimizing sequence $\left\{\left(u_{ n}, v_{n},w_{n} \right)\right\}$ for $\left.J\right|_{S(a,b,c) }$. Thus by Gagliardo-Nirenberg inequality Lemma \ref{GN}, we have that
	$$
	\begin{aligned}
		\left|u_{n}-\bar{u}\right|_{p} & \leq  C_{p}\left|\nabla\left(u_{n}-\bar{u}\right)\right|_{2}^{\gamma_{p}}\left|u_{n}-\bar{u}\right|_{2}^{1-\gamma_{p}}, \\
		\left|v_{n}-\bar{v}\right|_{q} & \leq  C_{q}\left|\nabla\left(v_{n}-\bar{v}\right)\right|_{2}^{\gamma_{q}}\left|v_{n}-\bar{v}\right|_{2}^{1-\gamma_{q}},\\
		\left|w_{n}-\bar{w}\right|_{r} & \leq  C_{r}\left|\nabla\left(w_{n}-\bar{w}\right)\right|_{2}^{\gamma_{r}}\left|w_{n}-\bar{w}\right|_{2}^{1-\gamma_{r}}.
	\end{aligned}
	$$
	Then by Lemma \ref{wei_shi_jing_xing} and the boundedness of $\{(u_n,v_n,w_n)\}$, $\left(u_n, v_n, w_n \right) \rightarrow(u, v, w)$ in $L^{p}\left(\mathbb{R}^{3}\right) \times L^{q}\left(\mathbb{R}^{3}\right) \times L^{r}\left(\mathbb{R}^{3}\right), 2 \leq p,q,r<6$. Hence from Lemma \ref{strong_convergence}, $\left(u_n, v_n, w_n \right) \rightarrow(u, v, w)$ in $H$. Furthermore, $(\bar{u}, \bar{v}, \bar{w})$ is a minimizer, i.e., $(\bar{u}, \bar{v},\bar{w}) \in S(a,b,c )$ and $J(\bar{u}, \bar{v},  \bar{w})=m(a,b,c )$. Thus there are $\lambda_{1,0},\lambda_{2,0},\lambda_{3,0} \in \mathbb{R}$ such that $(\bar{u}, \bar{v}, \bar{w})$ satisfies
	\begin{equation*}
			\begin{cases}
				 - \Delta \overline u  + \left( {{V_1}(x) + {\lambda _{1,0}}} \right)\overline u  = {\mu _1}{{\left| {\overline u } \right|}^{{p} - 2}}\overline u  + \beta \overline v \overline w ,& \text { in }\ {\mathbb{R}^3}, \\ 
				 - \Delta \overline v  + \left( {{V_2}(x) + {\lambda _{2,0}}} \right)\overline v  = {\mu _2}{{\left| {\overline v } \right|}^{{p} - 2}}\overline v  + \beta \overline u \overline w ,& \text { in }\ {\mathbb{R}^3}, \\ 
				 - \Delta \overline w  + \left( {{V_3}(x) + {\lambda _{3,0}}} \right)\overline w  = {\mu _3}{{\left| {\overline w } \right|}^{{p} - 2}}\overline w  + \beta \overline u \overline v ,& \text { in }\ {\mathbb{R}^3},
			\end{cases}
	\end{equation*}
	with $|\bar{u}|_{2}^{2}=a,|\bar{v}|_{2}^{2}=b,|\bar{w}|_{2}^{2}=c $, $\bar{u}, \bar{v},  \bar{w} \geq  0$. By the maximum principle, $\bar{u}, \bar{v}, \bar{w}>0$. Combine with Lemma \ref{feifulie}, we finish the proof of Theorem \ref{th2}.$\hfill\Box$

	\bibliographystyle{plain}
	\bibliography{lunwen2_infty}
\end{document}